\documentclass[reqno,12pt]{amsart}
\usepackage{amssymb,amsmath,amsfonts}
\usepackage[T1]{fontenc}
\usepackage[latin1]{inputenc}
\usepackage{xcolor}

\newtheorem{thm}{Theorem}[section]
\newtheorem{defi}[thm]{Definition}
\newtheorem{lem}[thm]{Lemma}

\newtheorem{prop}[thm]{Proposition}
\newtheorem{rmk}[thm]{Remark}

\usepackage{enumerate}
\usepackage{indentfirst}
\usepackage{hyperref}
\usepackage{color}
\everymath{\displaystyle}

\def\a1{HHH}

\begin{document}

\title[On a second critical value for the local existence in Lebesgue spaces]{On a second critical value for the local existence of solutions in Lebesgue spaces}

\author{Brandon Carhuas-Torre, Ricardo Castillo, Miguel Loayza}

\address{Ricardo Castillo. Departamento de Matem\'{a}tica, Facultad de Ciencias, Universidad del B\'{\i}o-B\'{\i}o, Avenida Collao 1202, Casilla 5-C, Concepci\'{o}n, B\'{\i}o B\'{\i}o, Chile}
\email{rcastillo@ubiobio.cl}

\address{Brandon Carhuas-Torre and Miguel Loayza\footnote{Corresponding Author}. Departamento de Matem\'{a}tica, Universidade Federal de Pernambuco, Recife, Brazil}
\email{brandon.carhuas@ufpe.br; miguel.loayza@ufpe.br}

\date{ \today }

\begin{abstract} 
We provide new conditions for the local existence  of solutions  to the time-weighted parabolic  equation $ u_t - \Delta u = h(t)f(u) \mbox{ in } \Omega \times (0,T),$ where $ \Omega $ is a arbitrary smooth domain, $f\in C(\mathbb{R})$, $h\in C([0,\infty))$ and $u(0)\in L^r(\Omega)$. As consequence of our results, considering a suitable behavior of the non-negative initial  data, we obtain a second critical value $\rho^\star = 2r/(p-1),$ when $f(u)=u^p$ and $p> 1 + 2r/N$,
which determines the existence (or not) of a local solution $u \in L^\infty((0,T), L^r(\Omega)).$ 


\end{abstract}
\maketitle

\section{Introduction}
Let $\Omega\subset \mathbb{R}^N$ be a domain (bounded or unbounded) with smooth boundary $\partial \Omega$ whenever it exists. The local existence of non-negative solutions of the semilinear problem
\begin{equation}\label{Eq1}
 \left\{\begin{array}{rclcl}
 u_{t}- \Delta u &= &f(u) &  \text{in}&  \Omega \times (0,T),   \\
 u&= & 0  &  \text{on}&  \partial \Omega \times (0,T),\\
 u(0)&= & u_{0}\geq 0 &  \text{in}&  \Omega,
      \end{array}\right.
\end{equation}
with  $f \in C([0,\infty))$ non-decreasing, $u_0\in L^r(\Omega), u_0\geq 0$ has been completely characterized by  Laister et al. \cite{Laister1} in the case that $\Omega $ is either bounded  or $\Omega = \mathbb{R}^N$. They showed the existence of a critical value  
\begin{equation}\label{Cr.uno}
p^\star=1+\frac{2r}{N}    
\end{equation}
so that 
\begin{itemize}
  \item When $\Omega$ is bounded, problem (\ref{Eq1}) has a local non-negative $L^r-$solution for every $u_0 \in L^{r}(\Omega), u_0\geq 0$  if and only if
  \begin{equation} \label{1TT}
  \limsup_{t \to \infty} t^{-p^{\star}}f(t)<\infty, \mbox{ if }r>1,
  \end{equation}
  \begin{equation} \label{2TT}
  \int_1^\infty \sigma^{-p^{\star}} F(\sigma)d\sigma <\infty,\mbox{ where } F(\sigma)= \sup_{1\leq t \leq \sigma} f(t)/t, \mbox{ if }r=1.
  \end{equation}
  \item  The same result holds in $\Omega=\mathbb{R}^N$ replacing conditions (\ref{1TT}) and (\ref{2TT}) by
$$
  \limsup_{t \to \infty} t^{-p^{\star}}f(t)<\infty \ \mbox{ and } \ \limsup_{t \to 0} f(t)/t < \infty,
$$
$$
  \int_1^\infty \sigma^{-p^{\star}} G(\sigma)d\sigma <\infty \ \mbox{ and } \ \limsup_{t \to 0} f(t)/t < \infty,
$$
respectively, where $G(\sigma)= \sup_{0< t \leq \sigma} f(t)/t.$
\end{itemize}
Moreover, when the condition (\ref{1TT}) does not hold, that is $\limsup_{t\to \infty} t^{-p^\star} f(t)=+\infty$, they  provide a suitable non-negative initial data $u_0 \in L^r( \Omega)$ such that the problem (\ref{Eq1}) does not admit a local non-negative  $L^r-$solution.  Thus,  naturally, the following two questions arise: what does happen with the local existence in this case? Is it possible to define a class of initial data where the local existence hold?  As far as we know there are few results about this and only in the case $f(s)=s^q$, $q>1$  see \cite{Haraux}, \cite{Miyamoto}. 

Henceforth, without loss of generality, we assume that $0\in \Omega$ and $B_a \subset \Omega$, where $B_a$ denotes the open ball with radius $a>0$ centered at $0$. In order to answer the questions raised above we consider initial data in the sets $\mathcal{I}^{\rho,r}$ and $\mathcal{I}_{\rho,r}$, where
$$
\mathcal{I}^{\rho,r} =   \{ \psi \in L^r (\Omega); \psi \geq 0  \mbox{ and }  |x|^{\rho}  \psi^r(x)   \leq  K ~\chi_{B_a}(x)  \mbox{ a.e. in  }  \Omega, \mbox{ for some }K, a>0 \}, \\
$$
 $$
\mathcal{I}_{\rho,r} =   \{ \psi \in L^r (\Omega); \psi \geq 0  \mbox{ and } |x |^{\rho} \psi^r(x) \geq  K \chi_{B_a}(x)\mbox{ a.e. in }   \Omega \mbox{ for some }K,a>0 \},
$$
for $ \rho >0$ and $r\geq1$. Here $\chi_{B_a}$ denotes the characteristic function  of the ball $B_a$.  As consequence of our result we find a new critical exponent when $f(u)=u^q, q>1$, see Remark \ref{Rem1} (i), which is given by
\begin{equation}\label{Cr.dos}
     \rho^\star = \frac{2r}{q-1}.
\end{equation}
We say that $\rho^*$ is the second critical value.


The sets $\mathcal{I}^{\rho,r}$ and $\mathcal{I}_{\rho,r}$  are inspired in the ones considered by Lee and Ni \cite{Lee} and Wang \cite{Wang}. In those works, the authors considered problem (\ref{Eq1}) for $f(u)=u^q, q>1,$ and initial data in $C(\mathbb{R}^N)\cap L^\infty(\mathbb{R}^N)$. For $q>q_F$, where $q_F=1+2/N$ is the Fujita exponent, they analyze the existence of global  and blow-up solutions when $u_0$ belongs to the sets  $I_\rho$ or $I^\rho$, where
$$
I_{\rho}=\{\psi \in C(\mathbb{R}^N), \psi\geq 0, \liminf_{|x|\to \infty} |x|^{\rho} \psi(x)>0 \},
$$
$$
I^{\rho}=\{\psi \in C(\mathbb{R}^N), \psi\geq 0, \limsup_{|x|\to \infty} |x|^{\rho} \psi(x)<\infty \}.
$$
They determined a critical value $\rho_F=2/(q-1)$, which is called the second critical value, such that if $0<\rho< \rho_F$ and $u_0\in I_\rho$ then the solution blows up in finite time, while if $\rho\geq \rho_F$ and $u_0\in I^{\rho}$ then the solution is global.

It is worth mentioning that initial data in the form $u_0=|\cdot |^{-\alpha} \chi_{B_a}$ was used firstly in \cite{Souplet2} to show the non-existence of non-negative solution for a  system related to (\ref{Eq1}), and  characterizations for parabolic problems related to (\ref{Eq1}), similar to the results obtained by \cite{Laister1}, have been obtained in \cite{Castillo2020}, \cite{Castillo2022} and for a fractional equation in \cite{Castillo2021}.

In our first result, we analyze the existence of non-negative $L^r-$solutions. These solutions are understood in the sense of Definition \ref{Def}.

\begin{thm}\label{Main1} Let  $f \in C([0,\infty))$ be convex,  non-decreasing and $f(0)=0$.
\begin{enumerate}[(i)]
  \item Assume there exists $ p_0 > p^\star$ such that $\limsup_{t \rightarrow \infty} t^{-p_0} f(t) < \infty$ .
  Let
  $$
  p_{\inf} = \inf \left\{p > p^\star; \limsup_{t \rightarrow \infty} t^{-p} f(t) < \infty \right \}
 $$
  and $0 < \rho < 2r/(p_{\inf} - 1).$ Then, for every nontrivial $u_0 \in \mathcal{I}^{\rho,r}$, problem $(\ref{Eq1})$ has a local non-negative $L^r-$solution on some interval $(0,T)$. Moreover, $u\in L^\infty_{loc}((0,T), L^\infty(\Omega))$ and $\|u(t)\|_{L^\infty}\leq Ct^{-\rho/2r}$ for all $t\in (0,T)$ and some constant $C>0$.
  
  \item  Assume that there exists $ p_0 > p^\star$ such that $\liminf_{t \rightarrow \infty} t^{-p_0} f(t) >0$. If
  $$
  p_{\sup} = \sup \left\{p > p^\star ; \ \liminf_{t \rightarrow \infty} t^{-p} f(t)>0 \right \}\leq \infty,
$$
  and $2r/(p_{\sup} - 1) < \rho < N$, then for every $u_0 \in \mathcal{I}_{\rho,r}$ problem $(\ref{Eq1})$ does not have a local non-negative $L^r-$solution.
  
\end{enumerate}
\end{thm}

\begin{rmk}\label{Rem1} Here are some comments about Theorem \ref{Main1}.
\begin{enumerate}[(i)]
  \item Suppose that $f(t) = t^q,$ with $q>p^\star$ and $p^\star$ given by (\ref{Cr.uno}). Then $p_{\sup}=p_{\inf}=q$ and Theorem \ref{Main1} implies that if $ 0 <\rho < \rho^\star $ and $u_0 \in \mathcal{I}^{\rho,r},$ then there exists a local non-negative $L^r-$solution of (\ref{Eq1}), and if  $\rho^\star < \rho < N$ and $u_0 \in \mathcal{I}_{\rho,r},$  then there is no a local non-negative $L^r-$solution of (\ref{Eq1}). Thus, the value $\rho^\star$, given by (\ref{Cr.dos}), is the second critical value for the existence of local non-negative $L^r-$solutions.

  \item Other convex and non-decreasing functions where it is possible to apply Theorem \ref{Main1} are:
$ f(t)= e^{\alpha t} -1$, with  $\alpha>0$  and $f(t)= (1+t)^{q} [\ln(1+t)]^s$ with  $s\geq1,  q>p^\star. $
In the first case $p_{\inf}$ does not exist but $p_{\sup}=\infty,$ and in the second case $p_{\sup}=p_{\inf}=q$.

  \item The convexity condition on $f$ and $f(0)=0$ is required in (i) to guarantee that the function $t\in (0,\infty)  \mapsto f(t)/t$ is well defined and non-decreasing, and in (ii) to  show the estimate $f(S(t)u_0) \leq S(t)f(u_0),$ see the proof of Proposition \ref{Prop1}. 

  \item  The condition $\liminf_{t \rightarrow \infty} t^{-p_0} f(t)>0$, for some $p_0 > p^\star$, implies that the one (\ref{1TT}) does not hold.  Indeed, we argue by contradiction and assume that (\ref{1TT}) holds. Then  $f\left ( t \right )\leq C t^{p^\star}$ for $t$ sufficiently large and some constant $C>0$. Thus, $t^{-p_0}f\left ( t \right )\leq Ct^ { p^\star -p_0}$ for $t$ large and we have  a contradiction letting $t\to \infty$.

\end{enumerate}
\end{rmk}

The arguments used to deal with problem (\ref{Eq1}) can also be used to address the time-weighted parabolic problem
\begin{equation}\label{Eqq2}
 \left\{\begin{array}{rclllc}
 u_{t}- \Delta u &= &h(t)  g(u) &  \text{in}&  \Omega \times (0,T),   \\
 u&= & 0 &  \text{in}&  \partial \Omega \times (0,T),\\
 u(0)&= & u_{0}  &  \text{in}&  \Omega,
      \end{array}\right.
\end{equation}
where $g \in C(\mathbb{R})$, $u_0 \in L^r(\Omega)$ without the signal restriction and a general weight function $h \in C([0,\infty))$.


Problem (\ref{Eqq2}), with $h(t)=1, g(u)=|u|^{p-1}u, p>1$, has been widely studied when either $\Omega= \mathbb{R}^N$ or $\Omega$ is a bounded domain. From the results of \cite{Brezis}, \cite{Celik}, \cite{Weissler1}, \cite{Weissler3} it is well known that: 
\begin{itemize}
    \item If either $p<p^\star$ and $r\geq 1$ or $p=p^\star$ and $r>1$, then (\ref{Eqq2}) admits a local $L^r-$solution.

    \item If either $p>p^\star$ and $r\geq 1$ or $p=p^\star$ and $r=1$ there exist $u_0\in L^r(\Omega), u_0\geq 0$ so that problem (\ref{Eqq2}) does not admit a nonnegative local $L^r-$solution.
\end{itemize}
Recently, Miyamoto \cite{Miyamoto} studied the so-called doubly critical case, that is, $p=p^\star, r=1$, considering $h(t)=1$ and the set
\begin{equation}\label{X.uno}
X_\alpha=\left \{ u_0\in L_{loc}^{1}\left ( \mathbb{R}^{N} \right );\int_{\mathbb{R}^{N}}^{}\left | u_0 \right |\left [ \log \left ( e+\left | u_0 \right | \right ) \right ]^{\alpha}dx< \infty  \right \}\subset L^{1}\left ( \mathbb{R}^{N} \right ).
\end{equation}
He showed:

\begin{itemize}
    \item  If $u_0 \in X_\alpha$ for some $\alpha \geq N/2$, then problem (\ref{Eqq2}) admits a local $L^1-$solution. 
    
    \item For each $0 \leq \alpha < N/2$ there exists a non-negative $u_0 \in X_\alpha$ so that the  problem (\ref{Eqq2}) does not have a local non-negative $L^1-$solution. 

\end{itemize}

Assuming that $g$ is  locally Lipschitz and $g(0)=0$ it is well defined the non-decreasing function $G:[0,\infty )\rightarrow [0,\infty )$ given by
 \begin{equation}\label{h1}
 G(s)= \sup_{0< |t| \leq s}\frac{g\left ( t \right )}{t}, \quad s>0; \quad G(0)=0.
 \end{equation}
This function was used by Laister and Sier\.{z}\k{e}ga \cite{Laister2} to show the existence of a local $L^1-$solution for problem (\ref{Eqq2}), with $h=1$ and $u_0\in L^1(\mathbb{R}^N)$, under the integral condition 
\begin{equation}\label{L1.uno}
 \int_1^\infty \sigma^{-(1+2/N)} G(\sigma) d\sigma<\infty.  
\end{equation}

As in Theorem \ref{Main1}, we analyze the problem (\ref{Eqq2}) for $u_0 \in \mathcal{I}^{\rho,r}$ and $u_0\in \mathcal{I}_{\rho,r}$. We obtain the following result.



\begin{thm}\label{Main2} 
 \begin{enumerate}[(i)]
          \item  Suppose that $g$ is locally lipschitz, non-decreasing  and $g(0)=0$. For every $u_0 \in L^r(\Omega)$, with $|u_0| \in \mathcal{I}^{\rho,r}$ and $0<\rho<N$,  problem $(\ref{Eqq2})$ has a local $L^r-$solution on some interval $(0,T)$ if the following conditions holds
          \begin{equation}\label{h4}
     \int_{1}^{\infty }\sigma^{-\left ( 1+2r/\rho  \right )}h\left (  (A C_0 K^{1/r})^{2r/\rho}\sigma^{-2r/\rho} \right )G\left ( \sigma \right )d\sigma< \infty,
 \end{equation}
     where $A>1$, $C_0$ is given by Lemma \ref{Lem2} (for $q_1=q_2=\infty$) and $K$ is the constant of $|u_0|$ in the set $\mathcal{I}^{\rho,r}$. Moreover, $u\in L^\infty_{loc}((0,T), L^\infty(\Omega))$ and  there exists a constant $C>0$ such that $t^{\rho/2r}\|u(t)\|_{\infty}\leq C$ for $t\in (0,T)$.

     \item Assume that the restriction $g|_{[0,\infty)}$ is non-decreasing, convex, $g(0)=0$  and
     $$ p_{\sup} =  \sup \left\{p > p^\star ; \ \liminf_{t \rightarrow \infty} t^{-p} g(t)>0 \right \} < \infty.$$ If there exists $0<\epsilon<p_{\sup}-1$ such that
   \begin{equation}\label{BLOWUP}
      \limsup _{t \longrightarrow 0^{+}}t^{-\frac{\rho }{2r}\left(p_{\sup}-1-\epsilon \right )}\int_{0}^{t}h\left( \sigma \right )d\sigma= \infty,
   \end{equation}
then for every $u_0 \in \mathcal{I}_{\rho,r}$, $0<\rho<N$, problem $(\ref{Eqq2})$ does not have a non-negative local  $L^r-$solution.

 \end{enumerate}
\end{thm}
\begin{rmk}\label{Rem2} Here are some comments about Theorem \ref{Main2}.
\begin{enumerate}[(i)]
\item From (\ref{L1.uno}) and (\ref{h4}) we see that 
$$\int_1^\infty \sigma^{-(1+2r/\rho)} G(\sigma)d\sigma < \int_1^\infty \sigma^{-1+2/N} G(\sigma) d\sigma,$$
when $r=1$ and $h=1$. Hence, under hypotheses of Theorem \ref{Main2} one can get a solution of problem (\ref{Eqq2}) for $u_0\in \mathcal{I}^{\rho,1}$ even if (\ref{L1.uno}) does not hold. 

\item If $g(t)=|t|^{q-1}t, t\in \mathbb{R}, q\geq p^\star$, then $G(s)=s^{q-1}$ and condition (\ref{h4}) is satisfied, for $h=1$, when $\rho(q-1)/2<r$. Hence, Theorem \ref{Main2} allows us to obtain a local $L^r-$solution for initial data in the set $\mathcal{I}^{\rho,r}$ when $1\leq r<N(q-1)/2$ and the doubly critical case $r=N(q-1)/2=1$. 

\item If $f\in C[0,\infty)$ is convex with $f(0)=0$, then defining $g(t)=0$ for $t\leq 0$ and $g(t)=f(t)$ for $t>0$ we have that $g$ is continuous and 
$$G(s)=\sup_{0<|t|\leq s} \frac{g(t)}{t}= \sup_{0<t\leq s} \frac{f(t)}{t}= \frac{f(s)}{s}.$$
Hence, by Theorem \ref{Main2}, assuming the hypotheses of Theorem \ref{Main1}, with $h(t)=t^{\beta}$, $\beta>-1$, we have:
\begin{itemize}
    \item If $u_0 \in \mathcal{I}^{\rho,r}$ and $0<\rho<2r(\beta+1)/(p_{\inf}-1)$ then problem (\ref{Eqq2}) admits a non-negative $L^r-$solution. 
    \item If $u_0\in \mathcal{I}_{\rho,r}$ and $2r(\beta+1)/(p_{\sup}-1)<\rho<N$, then problem (\ref{Eqq2}) does not admit a non-negative local $L^r-$solution.
\end{itemize}
In particular, when $\beta=0$ we recover Theorem \ref{Main1}, and for $f(t)=t^q, q>1$  the second critical value is given by $\rho^*=2r(\beta+1)/(q-1)$, see Remark \ref{Rem1}(i).

\end{enumerate}
\end{rmk}

Under an  integral condition stronger than (\ref{h4}) we are able to show the uniqueness of a local $L^r-$solutions of problem (\ref{Eqq2}). Given  $g\in C(\mathbb{R})$ locally lipschitz with $g(0)=0$, we define the non-decreasing function $L:[0,\infty )\rightarrow [0,\infty )$  by
$$
L\left ( s \right )= \sup _{\underset{u\neq v}{\left | u \right |,\left | v \right |\leq s}}\frac{g\left ( u \right )-g\left ( v \right )}{u-v}, \ s>0; \quad L(0)=0.
$$


Our uniqueness result is the following.

\begin{thm}\label{Main3} Let $g\in C(\mathbb{R})$ be  locally Lipschitz and $g(0)=0$. Problem (\ref{Eqq2}) admits an unique solution in the class
\begin{equation}\label{Con.uni}
\{ u\in L^{\infty }\left ( \left ( 0,T \right ),L^{r}\left ( \Omega  \right ) \right ); \sup_{t\in (0,T)} t^{\rho/2r}\|u(t)\|_{L^\infty}\leq C_1 \}
\end{equation}
if the following
 integral condition
\begin{equation}\label{def2}
    \int_{1}^{\infty }\sigma^{-\left ( 1+2r/\rho \right )}h\left (( C_1^{-1}\sigma)^{-2r/\rho} \right )L\left (  \sigma \right )d\sigma < \infty
\end{equation}
is satisfied.
\end{thm}
\begin{rmk} Here are some comments about Theorem \ref{Main3}.
\begin{enumerate}[(i)]
    \item From Theorem \ref{Main2}(i) the set (\ref{Con.uni}) is not empty if $g$ is non-decreasing and $|u_0|\in \mathcal{I}^{\rho,r}$ with $0<\rho<N$.  
    \item Suppose $g\left ( t \right )=|t|^{p-1} t$ for $t\in \mathbb{R}$, $p>1$ and $h=1$. Then $L(s)=ps^{p-1}$ and condition (\ref{def2}) is verified for $p<1+2r/\rho$. 

    \item For $h=1$ a similar uniqueness result for problem (\ref{Eqq2}) and $u_0\in L^1(\mathbb{R}^N)$ was obtained in \cite{Laister2} when $\int_1^\infty \sigma^{-(1+2/N)} L(\sigma) d\sigma<\infty$.
\end{enumerate}
\end{rmk}

\section{Preliminaries}
We denote by $P_\Omega(x,y;t) $ the Dirichlet heat kernel associated to the operator $\partial_t-\Delta_\Omega,$
where $-\Delta_\Omega$ is the Dirichlet Laplacian for the open set $\Omega \subset \mathbb{R}^N$, that is, 
\begin{equation}\label{Sem}
S(t) u_0(x)= \int_{\Omega} P_{\Omega}(x,y;t) u_0(y) dy,
\end{equation}
for $u_0\in L^1(\Omega)$ and $\{S(t)\}_{t\geq 0}$ the heat semigroup. It is very well known that
\begin{equation}\label{Kernel}
P_D(x,y;t) \leq  P_{E}(x,y;t) \leq P_{N}(x,y;t)
\end{equation}
for $x,y \in D \subset E,$ where $D$ and $E$ are open subsets of $\mathbb{R}^N,$ and $P_N$ is defined by $$P_N(x,y;t):= P_{\mathbb{R}^N}(x,y;t)= (4\pi t)^{-N/2} e^{- |x-y|^2/4t}.$$

The following lemma is used in the proof of the non-existence results.  Item (i) was established in \cite{Laister1}, and item (ii)  in \cite{Castillo2020}.
\begin{lem} \label{Lem1}  Let $l, \delta>0$ be such that $B_{l + 2 \delta} \subset \Omega$ and $0 <\gamma<N$. There exist constants $c_N, c'_N>0$, depending only on $N$ and $0<\gamma<N$, such that
\begin{enumerate}[(i)]
  \item $S(t) \chi_{l} \geq c_N l^{N} (l + \sqrt{t})^{-N} \chi_{l + \sqrt{t}}~,$ for all $0<t\leq\delta^2. $
  \vspace{2mm}
  \item $S(t) |\cdot|^{-\gamma} \chi_{l} \geq c'_N \, t^{-\frac{\gamma}{2}} \, \chi_{\sqrt{t}}~,$ for all $0<t\leq \min\{\delta^2,l^2 \}. $
  \vspace{2mm}
\end{enumerate}

\end{lem} 

The notion of $L^r-$solution that we use in the work is the following.
\begin{defi} Let $u_0\in L^{r}\left ( \Omega \right )$ with $1\leq r< \infty $,  $g\in C(\mathbb{R})$  and $h\in C([0,\infty))$. A measurable function $u$  defined a.e in $\Omega \times \left ( 0,T \right )$, for some $T>0$, is called a local  $L^r-$solution of problem (\ref{Eqq2}) if
\begin{equation}\label{eq2}
   u\left ( t \right )=S\left ( t \right )u_0+\int_{0}^{t}S\left ( t-\sigma \right )h\left ( \sigma \right )g\left ( u\left ( \sigma \right ) \right )d\sigma=\mathfrak{F}\left ( u,u_0 \right )
\end{equation}
a.e. in $\Omega \times \left ( 0,T \right )$ and $u\in L^{\infty }\left ( \left ( 0,T \right ),L^{r}\left ( \Omega  \right ) \right )$.

When $u_0\geq 0$ and $g|_{[0,\infty)}\in C([0,\infty))$ we say that $u$ is a  non-negative $L^r-$solution if (\ref{eq2}) holds and $u\in L^{\infty }\left ( \left ( 0,T \right ),L^{r}\left ( \Omega  \right ) \right )$.
\label{Def}
\end{defi}

\begin{defi} Let $u_0\in L^{r}\left ( \Omega \right )$ with $1\leq r< \infty $,  $g\in C(\mathbb{R})$  and $h\in C([0,\infty))$. We say that a function $w$ defined a.e. in $\Omega \times \left ( 0,T \right )$ is a supersolution  of (\ref{Eqq2}) if $w\in L^{\infty }\left ( \left ( 0,T \right ),L^{r}\left ( \Omega  \right ) \right )$ and 
$\mathfrak{F}\left ( w,u_0 \right )\leq w$ a.e. in $\Omega \times \left ( 0,T \right )$, where  $\mathfrak{F}$ is defined by (\ref{eq2}).

Subsolutions of (\ref{Eqq2}) are defined similarly with reverse inequality.
\end{defi}

The following result is a version obtained in the proof of \cite[Theorem 2]{Laister2}.
\begin{lem}\label{Lem22}  Assume that $g\in C(\mathbb{R})$ is  non-decreasing, $h \in C([0,\infty)),$ and  $u_0\in L^{r}\left ( \Omega  \right )$ with $1\leq r<\infty$. If  $\overline{u}$ and $\underline{u}$ are a  supersolution and a subsolution of problem (\ref{Eqq2}) in $\Omega \times \left ( 0,T \right )$ respectively such that $\overline{u}\geq \underline{u}$, then  there exists a solution $u$  of problem (\ref{Eqq2}) defined on $\Omega \times \left ( 0,T \right )$ such that $\underline{u}\leq u\leq \overline{u}$.
\end{lem}
\begin{proof} From (\ref{eq2}),  $\overline{u} \geq \mathfrak{F}\left ( \overline{u},u_0 \right )$, since $\overline{u}$ is a supersolution of (\ref{Eqq2}). Moreover, $\mathfrak{F}$ is
non-decreasing on $u$, since $g$ is non-decreasing, $h\geq 0$ and the monotonicity property of the heat semigroup $\{S(t)\}_{t\geq 0}$. Consider the sequence $\left \{ u^{n} \right \}_{n\geq 0}$, given by $u^0 = \overline{u}$, and $\mathfrak{F}(u^{n-1},u_0) =u^{n}$ for $n\geq 1$. Since $\mathfrak{F}$ is non-decreasing and $\underline{u}$ is a subsolution with $\underline{u}\leq \overline{u}$ we have that the sequence $\left \{ u^n \right \}_{n\geq 0}$ is non-increasing a.e. in $\Omega \times \left ( 0,T \right )$ and $ \overline{u} \geq  u^n \geq u^{n+1} \geq \underline{u}$. Let $u\left ( x,t \right )=\lim_{n\rightarrow \infty }u^n\left ( x,t \right )$,
whenever it exists. The continuity of $g$, the monotonicity of semigroup $\{ S(t)\}_{t\geq 0}$, and the dominated convergence theorem allow us to conclude that $ u = \lim_{n\rightarrow \infty }u^{n} =  \lim_{n\rightarrow \infty }\mathfrak{F}\left (u^{n-1} ,u_0 \right )=\mathfrak{F}\left ( u,u_0 \right )$. In addition, since $|u|\leq \max\{|\overline{u}|, |\underline{u}|\}$ and  $\overline{u}, \underline{u} \in L^\infty ((0, T), L^r(\Omega))$ we conclude that $u \in L^\infty ((0, T), L^r(\Omega))$.

\end{proof}

Let $\Omega\subset \mathbb{R}^N$ be a smooth domain (possibly unbounded). We recall the well-known smoothing effect of the heat semigroup on Lebesgue spaces, that is,
\begin{equation}\label{reg}
    \left \| S\left ( t \right )\psi  \right \|_{L^{q_2}}\leq \left ( 4\pi t \right )^{-\frac{N}{2}\left ( \frac{1}{q_1}-\frac{1}{q_2} \right )}\left \| \psi  \right \|_{L^{q_1}},
\end{equation}
for $1\leq q_1\leq q_2\leq \infty$, $t>0$ and $\psi \in L^{q_1}\left ( \Omega  \right )$ (see \cite{Brezis} Lemma 7]. We also use the following estimate which can be obtained of estimate (\ref{Kernel}) and \cite[Proposition 2.1]{Slimene}.
\begin{lem}\label{Lem2}
Let $\gamma \in \left ( 0,N \right )$, and let $q_1,q_2\in (1,\infty ]$ satisfy
$$0\leq \frac{1}{q_2}< \frac{\gamma }{N}+\frac{1}{q_1}<1.$$
There exists a constant $C>0$, depending on $N,\gamma,q_1$ and $q_2$ such that
$$
  \left \| S\left ( t \right )\left ( \left | \cdot  \right |^{-\gamma }\psi  \right ) \right \|_{L^{q_2}}\leq Ct^{-\frac{N}{2}\left ( \frac{1}{q_1}-\frac{1}{q_2} \right )-\frac{\gamma }{2}}\left \| \psi  \right \|_{L^{q_1}},
$$
for all $t>0$ and all $\psi \in L^{q_1}\left ( \Omega  \right )$.
\end{lem}
The following proposition is used in the proof of our nonexistence results. Its proof follows the ideas of \cite{Weissler1} and \cite{Weissler2} (see also Lemma 15.6 in \cite{Souplet}).

\begin{prop}\label{Prop1} Assume $h \in C([0,\infty))$, and $f\in C([0,\infty))$  non-decreasing,    convex, $f(0)=0$ and $f(s)>0$ for all $s\geq z_0$, for some $z_0\geq 0$, such that
\begin{equation}\label{Osgood}
\int_{z_0}^\infty \frac{d \sigma}{f(\sigma)} < \infty.
\end{equation}
Let $u_0:\Omega \to [0,\infty)$ and $u: \Omega \times (0,T) \rightarrow [0,\infty) $  measurable functions such that
\begin{equation}\label{Eq2}
u(t) \geq S(t)u_0 + \int_0^t S(t-\sigma)h(\sigma) f(u(\sigma)) d\sigma \  \ \ \ \ \mbox{ a.e. in } \Omega \times (0,T).
\end{equation}
Assume that $u(x, t) < \infty $  a.e. in $  \Omega \times (0,T)$ and $u_0\neq 0$. Then
\begin{equation}\label{Eq3}
\int_0^\tau h(\sigma) d\sigma \cdot \left(\int_{\|S(\tau) u_0 \|_{L^\infty}}^\infty \frac{d \sigma}{f(\sigma)}\right)^{-1} \leq 1
\end{equation}
for all $\tau \in (0,T).$
\end{prop}
\begin{proof} Since $f$ is convex, $f(0)=0$,  $0<\eta=\int_\Omega P_\Omega(x,y;t) dy \leq 1$ (see (\ref{Kernel})) using (\ref{Sem}) and Jensen's inequality we have
$$
\begin{array}{ll}
f (S(t)u)&=f\left (\int_\Omega p_\Omega(x,y;t)u(y)dy \right )\\
&=f\left (\eta \int_\Omega \frac{p_\Omega(x,y;t)}{\eta}u(y)dy \right ) \\
&\leq \eta f \left  (\int_\Omega \frac{p_\Omega(x,y;t)}{\eta}u(y)dy \right )\\
& \leq \int_{\Omega} p_\Omega(x,y;t) f(u(y))dy=S(t)f(u)
\end{array}
$$ for all $t>0$. From $(\ref{Eq2})$ and the semigroup properties we have
\begin{equation}\label{Non12}
S(s-t)u(t) \geq  \Theta(\cdot,t),
\end{equation}
for $t \in [0,s]$ with $s\in (0,T)$, where
$$\Theta(\cdot,t):= S(s)u_0 + \displaystyle\int_{0}^{t} h(\sigma) f(S(s - \sigma)u(\sigma))~d\sigma \leq u(\cdot,s) < \infty,$$
for all $(x,t) \in \Omega \times [0,s].$

Note that for $x \in \Omega$ fixed the function $\beta(t) := \Theta(x, t)$ is absolutely continuous on $[0,s] $, consequently it is differentiable a.e. in $[0,s]$  and from (\ref{Non12}) we have
\begin{equation}\label{Non22}
 \beta'(t) =h(t) f(S(s - t)u(x,t)) \geq h(t)f(\beta(t)) \ \ \ \ \mbox{ a.e. } t \in [0,s], \end{equation}
since that $f$ is non-decreasing.

For $z\geq z_0$, let $H\left ( z \right )=\int_{z}^{\infty }\frac{d\sigma }{f\left ( \sigma  \right )}$. From (\ref{Osgood}) and (\ref{Non22}), $ H'\left ( \beta \left ( t \right ) \right )=-\frac{\beta '\left ( t \right )}{f\left ( \beta \left ( t \right ) \right )}\leq -h\left ( t \right ).$
Integrating, we obtain $H\left ( \beta \left ( s \right ) \right )-H\left ( \beta \left ( 0 \right ) \right )\leq -\int_{0}^{s}h\left ( \mu  \right )d\mu$, that is,
\begin{eqnarray*}
	\int_{0}^{s}h\left ( \mu  \right )d\mu&\leq& \int_{\beta \left ( 0 \right )}^{+\infty }\frac{d\sigma }{f\left ( \sigma  \right )}-\int_{\beta \left ( s \right )}^{+\infty }\frac{d\sigma }{f\left ( \sigma  \right )}\leq\int_{\beta \left ( 0 \right )}^{+\infty }\frac{d\sigma }{f\left ( \sigma  \right )}.
\end{eqnarray*}
Thus
$$
\int_{ S(s)u_0 }^{\infty} \frac{d \sigma}{f(\sigma)} \geq  \int_{ \beta(0) }^{\infty} \frac{d \sigma}{f(\sigma)} \geq   \int_0^s h(\sigma) d\sigma
$$
and the results follows.
\end{proof}

\section{Proof of the results}
\textbf{Proof of Theorem \ref{Main1}-(i)} Since $f$ is convex and $f(0)=0$  the function  $t \in (0,\infty) \mapsto f(t)/t$ is non-decreasing. Thus,  the function $G:(0,\infty )\rightarrow [0,\infty )$ given by $G\left ( t \right )=f(t)/t$ is well defined.

Let $\epsilon>0$ and $p_0>p^\star$ such that $\limsup_{t\to \infty} t^{-p_o} f(t)<\infty$ and $p_0<p_{\inf}+\epsilon.$ There exists $t_0>1$ so that $f(t)\leq \eta t^{p_0}$ for some $\eta>0$ and all $t\geq t_0.$ Thus,
\begin{equation}\label{INT1}
	  \int_{1}^{\infty }\sigma^{-\left ( 1+\frac{2r}{\rho } \right )}G\left ( \sigma \right )d\sigma=C+\eta \int_{t_0}^ \infty \sigma^{-2 -\frac{2r}{\rho }+ p_0}d\sigma< \infty,
\end{equation}
since $p_0< p_{\inf}+\epsilon <1+2r/\rho$ if $\epsilon>0$ is sufficiently small.

Let $w\left ( t \right )=A S\left ( t \right )u_0$ with $A>1$, and $u_0\in \mathcal{I}^{\rho,r}$, that is, there is a constant $K>0$ such that $0\neq u_0\leq K^{1/r} \left | \cdot  \right |^{-\rho /r} \chi_{B_a}$. From Lemma \ref{Lem2} we have $\|S(t)u_0\|_{L^\infty} \leq C_0 K^{1/r}t^{-\rho/2r} $. Thus,
\begin{equation}\label{Blow6}
\begin{array}{lcll}
\int_0^t S\left ( t-\sigma \right )f\left ( w\left ( \sigma \right ) \right )d \sigma&=&\int_0^t S\left ( t-\sigma \right )f\left ( AS\left ( \sigma \right )u_0 \right ) d\sigma&\\
&=&\int_0^tS\left ( t-\sigma \right ) G\left ( AS\left ( \sigma \right )u_0 \right )AS\left ( \sigma \right )u_0d\sigma &\\
&\leq&  \int_0^t S\left ( t-\sigma \right ) G\left ( A C_0 K^{1/r}   \sigma^{-\rho/2r} \right )AS\left ( \sigma \right )u_0 d\sigma&\\
&\leq& AS\left ( t \right )u_0 \int_0^tG\left (AC_0K^{1/r} \sigma^{-\rho /2r} \right )d\sigma&\\
&=& w(t) (A C_0)^{2r/\rho} K^{2/\rho} \left (\frac{2r}{\rho}\right) \int_{t^{-\rho/2r} AC_0 K^{1/r}}^\infty \sigma^{-(1+\frac{2r}{\rho}) }G(\sigma) d\sigma.
\end{array}
\end{equation}
Then, from (\ref{INT1}) and (\ref{Blow6}) we have
\begin{eqnarray*}
\mathfrak{F}\left ( w,u_0 \right )&=&S\left ( t \right )u_0+\int_{0}^{t}S\left ( t-\sigma \right )f\left ( w\left ( \sigma \right ) \right )d\sigma\\
&\leq& S\left ( t \right )u_0+ w(t) (AC_0)^{2r/\rho} K^{2/\rho} \left (\frac{2r}{\rho}\right) \int_{t^{-\rho/2r} AC_0K^{1/r}}^\infty \sigma^{-(1+\frac{2r}{\rho}) }G(\sigma) d\sigma\\
&=& \left [ \frac{1}{A} + (AC_0)^{2r/\rho} K^{2/\rho} \left (\frac{2r}{\rho}\right ) \int_{t^{-\rho/2r} AC_0K^{1/r}}^\infty \sigma^{-(1+2r/\rho) }G(\sigma) d\sigma \right ] w\left ( t \right )\\
&\leq &w\left ( t \right ),
\end{eqnarray*}
for $t\in (0,T)$ with $T>0$ sufficiently small. Thus, $\mathfrak{F}\left ( w,u_0 \right )\leq w$ in $( 0,T)$ and $w$ is a supersolution of (\ref{Eq1}) in $(0,T)$.  

Since $f(0)=0$ we see that $v=0$ is a subsolution of (\ref{Eq1}). From Lemma \ref{Lem22}, with $h=1$, $g\in C(\mathbb{R})$ defined by $g(t)=0, t\leq 0$ and $g(t)=f(t)$ for $t>0$, we conclude that there exists a $L^r-$solution of problem (\ref{Eq1}) such that  $0=v\leq u\leq w$ in $(0,T)$. In addition,
$\left \| u(t) \right \|_{L^{\infty}}\leq \left \| w(t) \right \|_{L^{\infty}}\leq C t^{-\rho/2r}$
for $t\in(0,T)$ for some constant $C>0$.

\textbf{Proof of Theorem \ref{Main1}-(ii)} We divide the proof in two cases:

\textbf{Case 1. } $p_{\sup}<\infty$. Let $u_0 \in \mathcal{I}_{\rho,r}$, that is, $u_0\geq K^{1/r} |\cdot|^{-\rho/r} \chi_{B_a}=v_0$, where $B_{a} \subset \Omega$. Note that $v_0 \in L^r(\Omega)$ since $0<\rho<N$. 

We argue by contradiction and  suppose that there exists a local non-negative $L^r$-solution $u$ of (\ref{Eq1}) with initial condition $u_0 \in \mathcal{I}_{\rho,r}$, that is,
$$ u(t) = S(t)u_0 + \int_0^t S(t-\sigma) f(u(\sigma)) d\sigma,$$
for $t\in (0,T)$, $u \in L^\infty((0,T), L^r(\Omega))$ for some $T>0$. Hence 
\begin{equation}\label{Fi.uno}
u(t) \geq  S(t)v_0 + \int_0^t S(t-\sigma) f(u(\sigma)) d\sigma
\end{equation}
for $t\in (0,T)$.

 Note that Lemma \ref{Lem2} implies  $S(t) v_0 \in L^\infty(\Omega)$ for all $t>0$. Moreover,  by Lemma \ref{Lem1}, we have $S(t) v_0 \geq c_N' t^{-\frac{\rho}{2r}} \chi_{\sqrt{t}}$,
for all $0<t<\min\{(a/3)^2,T\}$. Thus
\begin{equation}\label{Blow2}
\|S(t) v_0\|_{L^\infty} \geq c_N' t^{-\rho/2r},
\end{equation}
for all $0<t<\min\{(a/3)^2,T\}.$

Let $\epsilon$ be sufficiently small such that $0<\epsilon<p_{\sup} - 1-2r/\rho$. There exists $p_0 >p^\star$ such that $\liminf_{t\rightarrow \infty} t^{-p_0} f(t)>0$ and $p_{\sup}<p_0+\epsilon$. Thus, there exists a constant $\eta>0$ such that $f(t) \geq \eta t^{p_0}$ for $t\geq t_0\geq 1$ for some constant $t_0$. From (\ref{Blow2}), we see that there exists $\tau_0$ such that $\|S(\tau)v_0\|_{L^\infty} \geq t_0$ for all $0<\tau\leq \tau_0$. Using again (\ref{Blow2})
$$
\begin{array}{ll}
\int_{\|S(\tau)v_0\|_{L^\infty}}^\infty \frac{d\sigma}{f(\sigma)}& \leq \frac{1}{\eta}\int_{\|S(\tau)v_0\|_{L^\infty}}^\infty \frac{d\sigma}{ \sigma^{p_0}} \\
&= \frac{1}{\eta(p_0-1)}\|S(\tau)v_0\|_{L^\infty}^{1-p_0}\\
&\leq \frac{1}{\eta(p_0-1)}\|S(\tau)v_0\|_{L^\infty}^{1-p_{\sup}+\epsilon}\\
&\leq \frac{(c_N')^{1-p_{\sup} +\epsilon}}{\eta(p_0-1)} \tau^{-\rho (1-p_{\sup} +\epsilon)/2r}
\end{array}
$$
Thus, from (\ref{Fi.uno}) and Proposition \ref{Prop1}
\begin{equation} \label{Blow4}
1\geq \tau \left(\int_{\|S(\tau) v_0\|_{L^\infty}}^{\infty}  \frac{d\sigma}{f(\sigma)}\right)^{-1} \geq \eta(p_0-1) (c_N')^{-(1-p_{\sup} +\epsilon)} \tau^{1+ \rho (1-p_{\sup} +\epsilon)/2r}
\end{equation}
for all $0<\tau\leq\min\{\tau_0, (a/3)^2, T\}$. This is a contradiction for $\tau$  sufficiently  small, since  $1+\rho(1-p_{\sup}+\epsilon)/2r<0$.

\medskip
\textbf{Case 2. } $p_{\sup}=\infty$. Let $u_0 \in \mathcal{I}_{\rho,r}$ with $\frac{2r}{p_{\sup }- 1}=0 < \rho < N$ and let $p_0>p^\star$ such that $2r/(p_0-1) < \rho < N$ such that $\liminf_{t \rightarrow \infty} t^{-p_0} f(t) >0.$ Arguing as the previous case, considering $p_0$ instead of $p_{\sup}$, the result follows.


\medskip
\textbf{Proof of Theorem \ref{Main2}} (i) We adapt the proof of Theorem \ref{Main1}-(i)  to obtain a solution for initial data with indefinite sign. To do this, we  show that  for $|u_0|\in \mathcal{I}^{\rho,r}$ and $A>1$, the functions $w(t)=A S(t)|u_0|$ and $v(t)=-AS(t)|u_0|$ are, respectively, a supersolution and a subsolution of (\ref{Eqq2}).

Note that taking $t=s$ in (\ref{h1}), we have
 \begin{equation}\label{h2}
    g(s)\leq sG(s), \quad s\geq 0,
 \end{equation}
 and taking $t=-s$,
 \begin{equation}\label{h3}
      g(s)\geq sG(-s), \quad s\leq 0.
 \end{equation}
Arguing similarly as in the derivation of (\ref{Blow6}),  using the estimate $\|S(\sigma)|u_0|\|_{L^\infty}\leq C_0K^{1/r}\sigma^{-\rho/2r}$ and inequality (\ref{h2})
\begin{equation}\label{Blow62}
\begin{array}{lcll}
\int_0^t  S\left ( t-\sigma \right) h(\sigma)g\left ( w\left ( \sigma \right ) \right )d\sigma&=&\int_0^t h(\sigma) S\left ( t-\sigma \right )g\left ( AS\left ( \sigma \right )|u_0| \right ) d\sigma&\\
&\leq &\int_0^t h(\sigma)S\left ( t-\sigma \right ) G\left ( AS\left ( \sigma \right )|u_0| \right )AS\left ( \sigma \right )|u_0|d\sigma&\\
&\leq& AS\left ( t \right )|u_0| \int_0^th(\sigma)G\left (A C_0 K^{1/r} \sigma^{-\rho/2r} \right )d\sigma&\\
&=& w(t) \int_0^th(\sigma)G\left (AC_0K^{1/r} \sigma^{-\rho/2r} \right )d\sigma.
\end{array}
\end{equation}
Hence, by  inequality (\ref{h4})
$$
\begin{array}{lll}
\mathfrak{F}\left ( w,u_0 \right )&=&S\left ( t \right )u_0+\int_{0}^{t}S\left ( t-\sigma \right )h(\sigma)g\left ( w\left ( \sigma \right ) \right )d\sigma\\
&\leq& S\left ( t \right )|u_0|+w(t) \int_0^th(\sigma)G\left (AK^{1/r} \sigma^{-\rho/2r} \right )d\sigma\\
&\leq& \left ( \frac{1}{A} + \int_0^th(\sigma)G\left (AC_0K^{1/r} \sigma^{-\rho/2r} \right )d\sigma \right )w\left ( t \right )\\
&\leq &w\left ( t \right ),
\end{array}
$$
for $t\in (0,T)$ with $T>0$ sufficiently small. Thus, $\mathfrak{F}\left ( w,u_0 \right )\leq w$ in $( 0,T)$ and $w$ is a supersolution of (\ref{Eqq2}) in $(0,T)$. 

 Now, using  (\ref{h3}) and arguing similarly as in the derivation of (\ref{Blow62}) we have
$$
\begin{array}{lcll}
\int_0^t S\left ( t-\sigma \right )h(\sigma)g\left ( v\left ( \sigma \right ) \right )d\sigma&=&\int_0^t S\left ( t-\sigma \right )h(\sigma)g\left ( -AS\left ( \sigma \right )|u_0| \right ) d\sigma&\\
&\geq&\int_0^t S\left ( t-\sigma \right )h(\sigma)G( AC_0K^{1/r}s^{-\rho/2r})[-AS\left ( \sigma \right )|u_0|] d\sigma&\\
&\geq& -AS\left ( t \right )|u_0| \int_0^th(\sigma)G\left (AC_0K^{1/r} \sigma^{-\rho/2r} \right )d\sigma&\\
&=& v(t)\int_0^th(\sigma)G\left (AC_0K^{1/r} \sigma^{-\rho/2r} \right)d\sigma
\end{array}
$$
and so
\begin{eqnarray*}
\mathfrak{F}\left ( v,u_0 \right )&=&S\left ( t \right )u_0+\int_{0}^{t}S\left ( t-\sigma \right )h(\sigma)g\left ( v\left ( \sigma \right ) \right )d\sigma\\
&\geq& -S\left ( t \right )|u_0|+v(t) \int_0^th(\sigma)G\left (AC_0K^{1/r} \sigma^{-\rho/2r} \right )d\sigma\\
&\geq& \left ( \frac{1}{A} + \int_0^th(\sigma)G\left (AC_0 K^{1/r} \sigma^{-\rho /2r} \right )d\sigma \right )v\left ( t \right )\\
&\geq &v\left ( t \right ),
\end{eqnarray*}
for $t\in (0,T)$ with $T>0$ sufficiently small. Thus, $\mathfrak{F}\left ( v,u_0 \right )\geq v$ in $( 0,T )$ and $v$ is a subsolution of (\ref{Eqq2}) in $(0,T)$.  

Therefore, Lemma \ref{Lem22} assures that problem (\ref{Eqq2}) admits a $L^r-$solution defined on $(0,T)$ and  $v(t)\leq u(t)\leq w(t)$ for $t\in (0,T)$. Hence, $|u(t)|\leq AS(t)|u_0|$  for $t\in (0,T)$, and 
$$\left \| u(t) \right \|_{L^{\infty}}\leq A \left  \|S(t)|u_0| \right \|_{L^{\infty}}\leq Ct^{-\rho/2r}$$
 for some constant $C>0$.

(ii) Let $u_0 \in \mathcal{I}_{\rho,r}$ with $ 0 < \rho < N$, that is, $u_0 \geq K |x|^{-\rho/r} \chi_a=v_0$ and $f=g|_{[0,\infty)}$. Then $v_0\in L^r(\Omega)$.

We arguing by contradiction and assume that there exists a local non-negative $L^r-$solution $u$ of (\ref{Eqq2}), defined on some interval $(0,T)$, with initial condition $u_0$. Since $u_0\geq v_0$ we obtain
$$ u(t) \geq S(t) v_0 + \int_0^t h(\sigma)S(t-\sigma) f(u(\sigma)) d\sigma, $$
for $t\in (0,T).$

Arguing as in the proof of Theorem \ref{Main1}(ii) we have
\begin{equation}\label{BLow2}
\|S(t) v_0\|_{L^\infty} \geq c_N' t^{-\rho/2r}
\end{equation}
for all $0<t<\min\{(a/3)^2,T\}$, see (\ref{Blow2}).

Since $p_{\sup}<\infty$ there exists $p_0>p_{\sup}-\epsilon$ and $\liminf_{t\to \infty} t^{-p_0}f(t)>0$. Thus, there exist $C>0$ and $t_0>1$ such that $f(t)\geq C  t^{p_{\sup}-\epsilon}$ for all $t>t_0$. Then, from (\ref{BLow2}) for $\tau>0$ small we have
\begin{eqnarray*}
C^{-1}(p_{\sup} -1-\epsilon)^{-1} (c_N')^{-(p_{\sup}-1-\epsilon)} \tau^{\rho(p_{\sup}-1-\epsilon)/2r} &\geq& C^{-1} (p_{\sup} -1-\epsilon)^{-1} \|S(\tau) v_0\|_{L^\infty}^{1-p_{\sup}+\epsilon} \\
&=& C^{-1} \int_{\|S(t)v_0\|_{L^\infty}}^{\infty}  \frac{d\sigma}{\sigma^{p_{\sup}-\epsilon}} \\
&\geq & \int_{\|S(\tau)v_0 \|_{L^\infty}}^{\infty}  \frac{d\sigma}{f(\sigma)}.
\end{eqnarray*}
 Thus, Proposition \ref{Prop1},
$$1\geq  \int_0^\tau h(\sigma) d\sigma \cdot \left(\int_{\|S(\tau)v_0\|_{L^\infty}}^{\infty}  \frac{d\sigma}{f(\sigma)}\right)^{-1} \geq C (p_{\sup} -1-\epsilon) ( c_N')^{-(1-p_{\sup}+\epsilon)} \, \tau^{-\frac{\rho(p_{\sup}-1-\epsilon)}{2r}} \cdot \int_0^\tau h(\sigma) d\sigma,$$
for all $\tau$ small enough. This contradicts condition (\ref{BLOWUP}).


\medskip
\textbf{Proof of Theorem \ref{Main3}} Suppose that problem (\ref{Eqq2}) has two local $L^r$-solutions $u$ and $v$ in the class (\ref{Con.uni}) defined on some interval $(0,T)$ with initial data $u_0$ and $v_0$ respectively, that is,
\begin{equation}\label{Blow7}
    u\left ( t \right )=S\left ( t \right )u_0+\int_{0}^{t}S\left ( t-\sigma \right )h\left ( \sigma \right )g\left ( u\left ( \sigma \right ) \right )d\sigma
\end{equation}
and
\begin{equation}\label{Blow8}
    v\left ( t \right )=S\left ( t \right )v_0+\int_{0}^{t}S\left ( t-\sigma \right )h\left ( \sigma \right )g\left ( v\left ( \sigma \right ) \right )d\sigma.
\end{equation}
Subtracting (\ref{Blow8}) from (\ref{Blow7}), and using estimate (\ref{reg}) we get 
$$\begin{array}{l}
\left\| u\left ( t \right )-v\left ( t \right )  \right \|_{L^{r}}\\
=\left \| S\left ( t \right )\left ( u_0-v_0 \right ) \right \|_{L^{r}}+\left\|\int_{0}^{t}S\left ( t-\sigma \right )h\left ( \sigma \right )\left [g\left ( u\left ( \sigma \right )\right ) -g\left ( v\left ( \sigma \right ) \right )    \right ]d\sigma \right \|_{L^{r}}\\
\leq \left \| u_0-v_0 \right \|_{L^{r}}+\int_{0}^{t}\left \| \frac{g\left ( u\left ( \sigma \right )\right ) -g\left ( v\left ( \sigma \right ) \right )}{u\left ( \sigma \right )-v\left ( \sigma \right )} \right \|_{L^\infty }h\left ( \sigma \right )\left \| S\left ( t-\sigma \right ) \left ( u\left ( \sigma \right )-v\left ( \sigma \right ) \right )\right \|_{L^{r}}d\sigma\\
\leq\left \| u_0-{v}_0 \right \|_{L^{r}}+ \int_{0}^{t}L\left ( C_1 \sigma^{-\rho/2r} \right )h\left ( \sigma \right )\left \| u\left ( \sigma \right )-v\left ( \sigma \right ) \right \|_{L^{r}}d\sigma.
\end{array}$$
By condition (\ref{def2}) and Gronwall's inequality we obtain
$$\|u(t)-v(t)\|_{L^r} \leq \|u_0-v_0\|_{L^r} \exp \left [\int_0^T L(C_1 \sigma^{-\rho/2r}) h(\sigma) d\sigma \right],$$ for $t\in (0,T)$. Hence, the uniqueness follows taking $u_0=v_0$.



\end{document}